\documentclass[a4paper]{amsart}
        \usepackage{latexsym}
        \usepackage{amssymb}
        \usepackage{amsmath}
        \usepackage{amsfonts}
        \usepackage{amsthm}
        \usepackage[hypertexnames=false]{hyperref}
        \ifpdf
          \usepackage[all,knot,poly,2cell]{xy}
        \else
          \usepackage[all,knot,poly,2cell,dvips]{xy}
        \fi
             \UseAllTwocells
        \usepackage{xspace}

        \usepackage{eucal}

        \theoremstyle{plain}
        \newtheorem{theorem}{Theorem}[section]
        \newtheorem{corollary}[theorem]{Corollary}
        \newtheorem{lemma}[theorem]{Lemma}
        \newtheorem{proposition}[theorem]{Proposition}
        \newtheorem{maintheorem}{Theorem}

        \theoremstyle{definition}
        \newtheorem{definition}[theorem]{Definition}

        \theoremstyle{remark}
        \newtheorem{remark}[theorem]{Remark}
        \newcommand{\suchthat}{\,:\,}
        \newcommand{\itemref}[1]{\eqref{#1}}

        \newcommand{\Z}{\mathbb{Z}}
        
        \newcommand{\F}{\mathbb{F}} 
        \newcommand{\Q}{\mathbb{Q}}

        \newcommand{\Orb}{\mathcal{O}}   
       \DeclareMathOperator{\spec}{Spec} 
        \newcommand{\red}[1]{{#1}_{\mathrm{red}}} 

           \newcommand{\lisset}{\mathrm{lis\text{\nobreakdash-}\acute{e}t}}

           \newcommand{\aff}{\mathrm{aff}}

        \newcommand{\MOD}{\mathsf{Mod}}    

         \DeclareMathOperator{\Hom}{Hom}

        \newcommand{\COHO}[1]{\mathcal{H}^{{#1}}}
        \newcommand{\trunc}[1]{\tau^{{#1}}}
        \newcommand{\RDERF}{\mathsf{R}}
        \newcommand{\LDERF}{\mathsf{L}}
        
        \newcommand{\DCAT}{\mathsf{D}}
        
        \newcommand{\SHom}{\mathcal{H}om}
        \newcommand{\SRHom}{\RDERF\SHom}
        \newcommand{\QCOH}{\mathsf{QCoh}}
        \newcommand{\COH}{\mathsf{Coh}}

        \renewcommand{\bar}[1]{\overline{{#1}}}

        \newcommand{\ID}[1]{\mathrm{Id}_{#1}}

        \newcommand{\tensor}{\otimes}
        
        \newcommand{\homotopic}{\simeq}

                \newcommand{\AB}{\mathsf{Ab}}

\renewcommand{\mid}{|}

\renewcommand{\subset}{\subseteq}
\numberwithin{equation}{section}

\newcommand{\qcsubscript}{\mathrm{qc}} 

\newcommand{\DQCOH}[1][]{\DCAT_{\qcsubscript{#1}}} 

\newcommand{\QCOHPSH}[1]{(#1_{\QCOH})_*} 
\newcommand{\MODPSH}[1]{(#1_{\lisset})_*} 

\newcommand{\Galpha}{\pmb{\alpha}} 
\newcommand{\GL}{\mathrm{GL}} 
\newcommand{\Gmu}{\pmb{\mu}} 
\newcommand{\Gm}{\mathbb{G}_m} 
\newcommand{\Ga}{\mathbb{G}_a} 

\DeclareMathOperator{\cohdim}{cd} 

\makeatletter
\newcommand{\labitem}[2]{%
\def\@itemlabel{(\textbf{#1})}
\item
\def\@currentlabel{\textbf{#1}}\label{#2}}
\makeatother
\usepackage{mathtools}
\usepackage{fixltx2e}[2005/12/01]
\newcommand{\Tr}{\mathrm{Tr}} 
\newcommand{\Rep}{\mathsf{Rep}}
\newcommand{\ant}{\mathrm{ant}} 

\title[Compact generation of derived categories of representations]{Algebraic groups and compact generation of their derived categories of representations}
\date{Dec 3, 2015}
\author[J. Hall]{Jack Hall}
\address{Mathematical Sciences Institute\\The Australian National
  University\\Acton ACT 2601\\Australia}
\email{jack.hall@anu.edu.au}
\author[D. Rydh]{David Rydh}
\address{KTH Royal Institute of Technology\\Department of Mathematics\\SE\nobreakdash-100\ 44\ Stockholm\\Sweden}
\email{dary@math.kth.se}
\thanks{This collaboration was supported by the G\"oran Gustafsson foundation.
The second author is also supported by the Swedish Research Council 2011-5599.}
\newcounter{saveenum}


\begin{document}
\begin{abstract}
  Let $k$ be a field. We characterize the group schemes $G$ over $k$, not necessarily affine, such that $\mathsf{D}_{\mathrm{qc}}(B_kG)$ is compactly generated. We also describe the algebraic stacks that have finite cohomological dimension in terms of their stabilizer groups.

\end{abstract}
\maketitle
\section*{Introduction}
In this article
we characterize two classes of group schemes over a field $k$:
\begin{enumerate}
\item those with compactly generated derived categories
of representations; and
\item those with finite (Hochschild) cohomological dimension.
\end{enumerate}
%
\subsection*{Compact generation}
Let $X$ be a quasi-compact and quasi-separated algebraic stack. Let $\DQCOH(X)$ be the unbounded derived category of lisse-\'etale $\Orb_X$-modules with
quasi-coherent cohomology sheaves. 

In~\cite{perfect_complexes_stacks}, we
showed that $\DQCOH(X)$ is compactly generated in many cases.
This does not always hold, however. With Neeman, we considered $B_k\Ga$---the
classifying stack of the additive group scheme over a field $k$---and proved
that every compact object of $\DQCOH(B_k\Ga)$ is $0$ if $k$ has positive
characteristic~\cite[Prop.~3.1]{hallj_neeman_dary_no_compacts}. In particular,
$\DQCOH(B_k\Ga)$ is not compactly generated.

If $\DQCOH(X)$ is compactly generated, then for every
point $x\colon \spec k\to X$ it follows that $\DQCOH(B_kG_x)$ is compactly
generated, where $G_x$ denotes the stabilizer group of $x$.
It follows that the presence of a $\Ga$ in a stabilizer group of positive
characteristic is an obstruction to compact
generation~\cite[Thm.~1.1]{hallj_neeman_dary_no_compacts}. We called such stacks 
\emph{poorly stabilized}. Our first main result is that this obstruction is the
only point-wise obstruction.
\begin{maintheorem}\label{MT:compact-gen-of-BG}
Let $k$ be a field, let $G$ be a group scheme of finite type over $k$
and let $\overline{G}=G\otimes_k \overline{k}$. Then
$\DQCOH(B_kG)$ is compactly generated if and only if
\begin{enumerate}
\item $k$ has characteristic zero or
\item $k$ has positive characteristic and the reduced connected component
  $\red{\overline{G}}^0$ is semi-abelian.
\end{enumerate}
Moreover, if $\DQCOH(B_kG)$ is compactly generated, then it is compactly generated by
\begin{enumerate}
\renewcommand{\theenumi}{\alph{enumi}}
\item\label{MT:compact-gen-of-BG:one} a single perfect complex if and only
  if the affinization of $\red{\bar{G}}^0$ is unipotent (e.g., if $G$
  is proper or unipotent); or
\item\label{MT:compact-gen-of-BG:comp-reps} the set of
  $k$-representations of $G$ that have compact image in $\DQCOH(B_kG)$
  when $G$ is affine; or
\item\label{MT:compact-gen-of-BG:irreps-fcd} the set of irreducible
  $k$-representations of $G$ when $G$ is affine and $k$ has
  characteristic zero or $G$ is linearly reductive.
\end{enumerate}
\end{maintheorem}
A group scheme is \emph{semi-abelian} if it is an extension of an abelian variety
by a torus (e.g., a torus or an abelian variety).
Note that $\red{\overline{G}}^0$ is semi-abelian precisely when there is
no subgroup
$\Ga\hookrightarrow \overline{G}$~\cite[Lem.~4.1]{hallj_neeman_dary_no_compacts}. The \emph{affinization} of a group scheme $G$ is the affine group scheme $\spec \Gamma(G,\Orb_G)$, see \cite[III.3.8]{MR0302656}. 

Recall that the abelian category $\QCOH(B_kG)$ is naturally identified with the
category $\Rep_k(G)$ of $k$-linear, locally finite representations of $G$. An \emph{irreducible} $k$-representation of $G$ is
a simple object of the abelian category $\Rep_k(G)$.
There is a natural functor
\[
\Psi_{B_k G}\colon \DCAT(\Rep_k(G))=\DCAT(\QCOH(B_kG))\to \DQCOH(B_kG).
\]
When $G$ is affine and $\DQCOH(B_kG)$ is compactly generated, then
$\Psi_{B_kG}$ is an equivalence~\cite[Thm.~1.2]{hallj_neeman_dary_no_compacts}.
Conversely, if $G$ is affine and $\DQCOH(B_kG)$ is not compactly generated,
then $G$ is poor (Theorem~\ref{MT:compact-gen-of-BG}) and $\Psi_{B_kG}$ is not an
equivalence~\cite[Thm.~1.3]{hallj_neeman_dary_no_compacts}. If $G$ is not
affine, then $\Psi_{B_kG}$ is not even full on bounded objects.
Nonetheless, $\DQCOH(B_kG)$ remains preferable. For example, $\DQCOH(B_kG)$ is always left-complete, which is not true of $\DCAT(\QCOH(B_kG))$; see \cite{hallj_neeman_dary_no_compacts}.

By Theorem \ref{MT:compact-gen-of-BG}\itemref{MT:compact-gen-of-BG:irreps-fcd}, if $G$ is linearly reductive, then $\DQCOH(B_kG)$ is compactly generated by the finite-dimensional irreducible $k$-representations of $G$. Since $\Rep_k(G)$ is a semisimple abelian category, $\Rep_k(G)$ is generated by the finite-dimensional irreducible $k$-representations.

Theorem \ref{MT:compact-gen-of-BG}\itemref{MT:compact-gen-of-BG:irreps-fcd} also implies that $\DQCOH(B_kG)$ is
compactly generated by $\Orb_{B_kG}$ when $G$ is unipotent and $k$
has characteristic zero. We wish to point out, however, that the abelian category
$\Rep_k(G)$ is not generated by the trivial one-dimensional
representation \cite[Cor.~3.4]{2013arXiv1306.5418G}. This further emphasizes the benefits of the derived category $\DQCOH(B_kG)$ over the abelian category $\Rep_k(G)$.

Theorem \ref{MT:compact-gen-of-BG}\itemref{MT:compact-gen-of-BG:irreps-fcd} cannot be extended to the situation where $B_kG$ is not of finite cohomological dimension (e.g., it fails for $k=\bar{\F}_2$ and $G=(\Z/2\Z)_k$). To prove Theorem \ref{MT:compact-gen-of-BG}, we explicitly describe a set of generators (Remark \ref{rem:gens-of-BG}).
\subsection*{Finite cohomological dimension}
Let $X$ be a quasi-compact and quasi-separated algebraic stack. An object of $\DQCOH(X)$ is \emph{perfect} if it is smooth-locally isomorphic to a bounded complex of free $\Orb_X$-modules of finite rank. While every
compact object of $\DQCOH(X)$ is perfect \cite[Lem.~4.4 (1)]{perfect_complexes_stacks}, there exist non-compact perfect complexes (e.g., $\Orb_{X}$, where $X=B_{\F_2}(\Z/2\Z)$). The following, however, are equivalent \cite[Rem.~4.6]{perfect_complexes_stacks}:
\begin{itemize}
\item every perfect object of $\DQCOH(X)$ is compact;
\item the structure sheaf $\Orb_X$ is compact;
\item there exists an integer $d_0$ such that for every quasi-coherent
  sheaf $F$ on $X$, the cohomology groups $H^{d}(X,F)$ vanish for
  all $d>d_0$; and
\item the derived global section functor $\RDERF\Gamma\colon \DQCOH(X)\to
  \DCAT(\AB)$ commutes with small coproducts.
\end{itemize}
We say that the stack $X$ has \emph{finite cohomological dimension} when it satisfies any of the conditions above. 

In the relative situation, the cohomological dimension of a morphism
depends in a subtle way on the separation properties of the target
(see Remark \ref{rem:strange_cd}). For this reason, in
\cite{perfect_complexes_stacks}, we introduced the more robust
notion of a \emph{concentrated} morphism. In the absolute situation, these two notions coincide, and we will use them interchangeably. 

If $G$ is a group scheme over a field $k$, a basic question to
consider is when its classifying stack $B_kG$ is concentrated. In characteristic $p>0$, the presence of unipotent subgroups of $G$ (e.g., $\Z/p\Z$, $\Galpha_p$, or $\Ga$) is an immediate obstruction. This
rules out all non-affine group schemes and $\GL_n$, where $n>1$. In
characteristic zero, if $G$ is affine, then its classifying stack is
concentrated. It was surprising to us that in characteristic zero, there are
non-affine group schemes whose classifying stack is concentrated. This
follows from a recent result of Brion on the coherent cohomology of anti-affine
group schemes~\cite{MR3102962}. More precisely,
we have the following theorem.

\begin{maintheorem}\label{MT:fin-coh-dim-of-BG}
Let $k$ be a field, let $G$ be a group scheme of finite type over $k$
and let $\overline{G}=G\otimes_k \overline{k}$. Then
$B_kG$ is concentrated if and only if
\begin{enumerate}
\item $k$ has positive characteristic and $\overline{G}$ is affine and
  linearly reductive; or
\item $k$ has characteristic zero and $\overline{G}$ is affine; or
\item $k$ has characteristic zero and the anti-affine part
  $G_\ant$ of $\bar{G}$ is of the form $G_\ant=S\times_A E(A)$, where
  $A$ is an abelian variety, $S\to A$ is an extension by a torus and
  $E(A)\to A$ is the universal vector extension.
\end{enumerate}
\end{maintheorem}

Finally, from Theorem~\ref{MT:fin-coh-dim-of-BG} using stratifications and
approximation techniques, we obtain a criterion for a stack to be concentrated.

\begin{maintheorem}\label{MT:concentrated-stacks}
Let $X$ be a quasi-compact and quasi-separated algebraic stack.
Consider the following conditions:
\begin{enumerate}
\item $X$ is concentrated.
\item Every residual gerbe $\mathcal{G}$ of $X$ is concentrated.
\item For every point $x\colon \spec k\to X$, the stabilizer group scheme
  $G_x$ is as in Theorem~\ref{MT:fin-coh-dim-of-BG}.
\end{enumerate}
Then (1)$\implies$(2)$\iff$(3). If $X$ has affine stabilizer groups and either equal
characteristic or finitely presented inertia, then (3)$\implies$(1).
\end{maintheorem}

Theorem~\ref{MT:concentrated-stacks} generalizes a result of Drinfeld and
Gaitsgory~\cite[Thm.~1.4.2]{MR3037900}: in characteristic zero, every quasi-compact and quasi-separated algebraic stack with finitely presented inertia and
affine stabilizers is concentrated.
Our generalization is made possible by a recent approximation result of the second author \cite{rydh-2014}. 

As an application of Theorem~\ref{MT:concentrated-stacks} and
\cite[Thm.~C]{perfect_complexes_stacks}, we obtain the following variant of
\cite[Thm.~B]{perfect_complexes_stacks} in positive characteristic:
\begin{maintheorem}\label{MT:globaltype_pos-char}
  Let $X$ be an algebraic stack of equal characteristic.
  Suppose that there exists a faithfully flat,
  representable, separated and quasi-finite morphism $X'\to X$ of finite presentation such
  that $X'$ has the resolution property and affine linearly reductive
  stabilizers. Then the unbounded derived category $\DQCOH(X)$ is compactly
  generated by a countable set of perfect complexes. In particular, this
  holds for every stack $X$ of s-global type with linearly reductive
  stabilizers.
\end{maintheorem}
\begin{proof}
  Argue exactly as in the proof of \cite[Thm.~B]{perfect_complexes_stacks} in
  \cite[\S 9]{perfect_complexes_stacks}: by
  \cite[Ex.~8.9]{perfect_complexes_stacks} and
  Theorem~\ref{MT:concentrated-stacks}, the stack $X'$ is $\aleph_0$-crisp,
  hence so is $X$ by \cite[Thm.~C]{perfect_complexes_stacks}.
\end{proof}
\subsection*{Acknowledgements}
It is our pleasure to thank Dan Petersen and Brian Conrad for fruitful
discussions concerning Section~\ref{S:coh-dim-BG}. We would also like
to thank Amnon Neeman for his encouragement with Appendix~\ref{A:duality}. 
Finally, we would like to thank the anonymous referee for a number of supportive comments. 
\section{Cohomological dimension of classifying stacks}\label{S:coh-dim-BG}
Let $G$ be a group scheme of finite type over a field $k$. In this section,
we give a complete classification of the groups $G$ such that $BG$ has finite
cohomological dimension (Theorem~\ref{MT:fin-coh-dim-of-BG}). In positive
characteristic, these are the linearly reductive groups
(Theorem~\ref{T:Nagata}). In characteristic zero, these are the affine groups
as well as certain groups built up from the universal vector extension of
an abelian variety (Theorem~\ref{T:conc-BG-char-zero}).

\begin{definition}
Let $G$ be an \emph{affine} group scheme over a field $k$ of characteristic~$p$.
We say that $G$ is
\begin{itemize}
\item \emph{nice} if the connected component of the identity $G^0$ is of
  multiplicative type
  and the number of geometric components of $G$ is not divisible by $p$; or
\item \emph{reductive} if the unipotent radical of $G_{\overline{k}}$ is
  trivial ($G$ not necessarily connected); or
\item \emph{linearly reductive} if every finite dimensional representation
  of $G$ is semi-simple, or equivalently, if $BG\to \spec k$ has cohomological
  dimension zero.
\end{itemize}
\end{definition}

Note that subgroups, quotients and extensions of nice group schemes are
nice. Indeed, this follows from the corresponding fact for connected
group schemes of multiplicative type~\cite[Exp.~IX, Props.~8.1,
  8.2]{MR0274459}.
Also note that if $G$ is nice, then $G^0$ is a twisted form of
$(\Gm)^n\times \Gmu_{p^{r_1}}\times\dots\times \Gmu_{p^{r_m}}$ for some
tuple of natural numbers $n,r_1,r_2,\dots,r_m$.

If $G$ is a group scheme of finite type over a field $k$, then there
is always a smallest normal subgroup scheme $G_\ant$ such that
$G/G_\ant$ is affine. The subgroup $G_\ant$ is anti-affine, that is, $\Gamma(G_\ant,\Orb_{G_\ant})=k$. Anti-affine groups
are always smooth, connected and commutative. Their structure has also been
described by Brion~\cite{MR2488561}.

In positive characteristic, we have the following result, which is classical
when $G$ is smooth and affine.

\begin{theorem}[Nagata's theorem]\label{T:Nagata}
Let $G$ be a group scheme of finite type over a field $k$. Consider the
following conditions:
\begin{enumerate}
\item $G$ is nice.\label{TI:Nagata:nice}
\item $G$ is affine and linearly reductive.\label{TI:Nagata:lin-red}
\item $BG$ has cohomological dimension $0$.\label{TI:Nagata:coh-dim-0}
\item $BG$ has finite cohomological dimension.\label{TI:Nagata:coh-dim-fin}
\end{enumerate}
Then \itemref{TI:Nagata:nice}$\implies$\itemref{TI:Nagata:lin-red}$\implies$\itemref{TI:Nagata:coh-dim-0}$\implies$\itemref{TI:Nagata:coh-dim-fin}.
If $k$ has positive
characteristic, then all four conditions are equivalent.
\end{theorem}
\begin{proof}
  First, recall that group schemes of multiplicative type are linearly reductive. Moreover, a finite \'etale group scheme is linearly reductive if and only if
  the number of geometric components is prime to the characteristic $p$ (by
  Maschke's Lemma and the fact that $\Z/p\Z$ is not linearly reductive).

  \itemref{TI:Nagata:nice}$\implies$\itemref{TI:Nagata:lin-red}: if $G$ is nice, then $G^0$ and $\pi_0(G)=G/G^0$ are linearly
  reductive group schemes; thus, so is $G$ (Lemma~\ref{lem:jarod}\itemref{lem:jarod:ext}).

  \itemref{TI:Nagata:lin-red}$\implies$\itemref{TI:Nagata:coh-dim-0}: that an affine group
  scheme $G$ is linearly reductive if and only if the classifying stack
  $BG$ has cohomological dimension
  $0$ is well-known.

  Now, suppose that $k$ has positive characteristic.
  That \itemref{TI:Nagata:lin-red}$\implies$\itemref{TI:Nagata:nice}
  when $G$ is smooth is Nagata's
  theorem~\cite{MR0142667}.
  That \itemref{TI:Nagata:lin-red}$\implies$\itemref{TI:Nagata:nice}
  in general is proved
  in~\cite[IV, \S3, Thm.~3.6]{MR0302656}. Let us briefly indicate how a
  similar argument proves that
  \itemref{TI:Nagata:coh-dim-fin}$\implies$\itemref{TI:Nagata:nice}.
  Assume that $BG$ has finite cohomological dimension.
  Then the same is true of $BH$ for every subgroup $H$ of $G$. In particular,
  there cannot be any subgroups of $G$ isomorphic to $\Z/p\Z$ or $\Galpha_p$.

  For the moment, assume that $G$ is affine. If $G$ is connected,
  then $G$ is of multiplicative type since $G$ has no subgroups
  isomorphic to $\Galpha_p$~\cite[IV, \S3, Lem.~3.7]{MR0302656}.
  If $G$ is disconnected, then the
  connected component $G^0$ has finite cohomological dimension and is
  thus of multiplicative type by the previous case. It follows that
  $\pi_0(G)$ has finite cohomological dimension
  (Lemma~\ref{lem:jarod}\itemref{lem:jarod:quot}).
  In particular, the rank has to be prime to $p$; hence $G$ is nice.

  Finally, suppose that $G$ is not affine. Since we are in positive
  characteristic, $G_\ant$ is semi-abelian, i.e., the extension of an abelian
  variety $A$ by a torus $T$~\cite[Prop.~2.2]{MR2488561}. In particular, the classifying
  stack $BA$ has finite cohomological dimension. Indeed, $A=G_\ant/T$ and
  $BT$ has cohomological dimension zero; then apply 
  Lemma~\ref{lem:jarod}\itemref{lem:jarod:quot}. The subgroup scheme
  $A[p]\subset A$ of $p$-torsion points is finite of degree $p^{2g}$, where $g$
  is the dimension of $A$. By assumption, $A[p]$ has finite cohomological
  dimension, so $A[p]$ is of multiplicative type. But this is
  impossible: the Cartier dual is $A^\vee[p]$, which is not \'etale.
\end{proof}
Let $f\colon X \to Y$ be a quasi-compact and quasi-separated morphism of algebraic stacks. Define $\cohdim(f)$, the \emph{cohomological dimension} of $f$, to be the least non-negative integer $n$ such that $\RDERF^d f_*M = 0$ for every $d>n$ and quasi-coherent sheaf $M$ on $X$. If no such $n$ exists, then we set $n=\infty$. We define the cohomological dimension of an algebraic stack $X$, $\cohdim(X)$, to be the non-negative integer $\cohdim(X \to \spec \Z)$. 

The lemma that follows is a simple refinement of \cite[Prop.~12.17]{2008arXiv0804.2242A}.
\begin{lemma}\label{lem:jarod}
Let $H \hookrightarrow G$ be an inclusion of group schemes of finite type over a field $k$ with quotient $Q$.
\begin{enumerate}
\item\label{lem:jarod:sub}
  Then $\cohdim(BH)\leq \cohdim(BG)+\cohdim(Q)$.
\setcounter{saveenum}{\value{enumi}}
\end{enumerate}
In addition, if $H$ is a normal subgroup scheme of $G$, then $Q$ is a group scheme of finite type over $k$ and the following holds:
\begin{enumerate}
  \setcounter{enumi}{\value{saveenum}}
\item\label{lem:jarod:ext}
  $\cohdim(BG)\leq \cohdim(BH)+\cohdim(BQ)$; and
\item\label{lem:jarod:quot}
  if $\cohdim(BH)=0$, then $\cohdim(BG)=\cohdim(BQ)$.
\end{enumerate}
\end{lemma}
\begin{proof}
Let $i\colon BH\to BG$ denote the induced morphism.
For \itemref{lem:jarod:sub}, by \cite[Lem.~2.2(4)]{perfect_complexes_stacks},
$\cohdim(BH) \leq \cohdim(BG) + \cohdim(i)$. Also, the pull-back of $i$ along
the universal $G$-torsor is $Q \to \spec k$. By
\cite[Lem.~2.2(2)]{perfect_complexes_stacks}, $\cohdim(i) \leq \cohdim(Q)$; the
claim follows.

For \itemref{lem:jarod:ext}, by \cite[Lem.~2.2(4)]{perfect_complexes_stacks},
$\cohdim(BG) \leq \cohdim(BQ) + \cohdim(j)$, where $j\colon BG\to BQ$ is the
induced morphism. Since $BH\to \spec k$ is a pull-back of $j$, it follows that
$\cohdim(j)\leq\cohdim(BH)$ \cite[Lem.~2.2(2)]{perfect_complexes_stacks}; the
claim follows.

For \itemref{lem:jarod:quot}, by \itemref{lem:jarod:ext}, we know that
$\cohdim(BG) \leq \cohdim(BQ)$. The reverse inequality follows from the
observation that the underived adjunction map $\ID{BQ}\to j_*j^*$ is an isomorphism  and
$\cohdim(j) = 0$.
\end{proof}

In characteristic zero, we have the following result.

\begin{theorem}\label{T:conc-BG-char-zero}
Let $G$ be a group scheme of finite type over a field $k$ of
characteristic zero. Then $BG$ has finite cohomological dimension if and only
if
\begin{enumerate}
\item $G$ is affine, i.e., $G_\ant$ is trivial; or
\item $G_\ant$ is of the form $G_\ant=S\times_A E(A)$, where
  $S$ is the extension of an abelian variety $A$ by a torus and $E(A)$ is
  the universal vector extension of $A$.
\end{enumerate}
\end{theorem}
\begin{proof}
By Lemma~\ref{lem:jarod}\itemref{lem:jarod:sub}--\itemref{lem:jarod:ext}, it is
enough to treat the cases where $G$ is either affine or anti-affine. If $G$ is
affine, then $G$ is a closed subgroup of $\GL_n$ for some~$n$. The induced
morphism $BG\to B\GL_n$ is a $\GL_n/G$-fibration. Since $\cohdim(B\GL_n)=0$ in
characteristic zero, it follows that $\cohdim(BG)\leq \cohdim(\GL_n/G)$ which is
finite. In the anti-affine case, the result follows from Proposition \ref{prop:coh-dim-of-anti-affine}. 
\end{proof}

\begin{proposition}\label{prop:coh-dim-of-anti-affine}
Let $G$ be a non-trivial anti-affine group scheme of finite type over a
field~$k$. If $k$
has characteristic zero and $G=S\times_A E(A)$, then $BG$ has cohomological
dimension zero. If not, then $BG$ has infinite cohomological dimension.
\end{proposition}
\begin{proof}
We have already seen that $BG$ has infinite cohomological dimension in positive
characteristic, so we may assume henceforth that $k$ has characteristic zero.  

By
Chevalley's Theorem \cite[Thm.~1.1]{MR1906417}, $G$ is an extension of an abelian variety $A$ by an
affine connected group scheme $G_\aff$. Since $G$ is commutative,
$G_\aff=T\times U$, where $T$ is a torus and $U$ is
connected, unipotent and commutative; in particular, $U\cong (\Ga)^n$ for some $n$.
Moreover, both the semi-abelian variety $S=G/U$ and the vector extension
$E=G/T$ are anti-affine, and $G=S\times_A E$~\cite[Prop.~2.5]{MR2488561}.
Since $T$ is linearly reductive, the cohomological dimension of
$B(G/T)$ equals the cohomological dimension of $BG$ (Lemma \ref{lem:jarod}\itemref{lem:jarod:quot}). We may thus assume that
$T=0$, so that $G=E$ is an extension of $A$ by $U$.
Let $g$ be the dimension of $A$ and let $n$ be the dimension of $U$.

Brion has calculated the coherent cohomology of $G$~\cite[Prop.~4.3]{MR3102962}:
\[
H^*(G,\Orb_G)=\textstyle \bigwedge^*(W^\vee),
\]
where $W\subseteq H^1(A,\Orb_A)^\vee$ is a $k$-vector space of dimension $g-n$.
If $g=n$, then $G$ equals the
universal vector extension $E(A)$ and $G$ has no non-trivial cohomology.

We now proceed to calculate $H^*(BG,\Orb_{BG})$ via the Leray spectral sequence for
the composition of $f\colon \spec k\to BG$ and $\pi\colon BG\to \spec k$.
Some preliminary observations.
\begin{enumerate}
\item Since $G$ is anti-affine, every coherent sheaf on $BG$ is a trivial
vector bundle.
\item If $G$ was assumed to be an affine group scheme, then the natural
  functor $\Psi^+\colon \DCAT^+(\QCOH(BG))\to \DQCOH^+(BG)$ is an equivalence of
  categories and the derived functor $\RDERF \MODPSH{f}\colon \DQCOH^+(\spec k)\to
  \DQCOH^+(BG)$ equals the composition of $\RDERF \QCOHPSH{f}\colon
  \DCAT^+(\MOD(k))\to \DCAT^+(\QCOH(BG))$ with $\Psi^+$. When $G$ is not affine,
  as in our case, both of these facts may fail.
\end{enumerate}
First consider $\mathcal{H}^i\bigl(\RDERF \MODPSH{f} k\bigr)=\RDERF^if_* k\in \QCOH(BG)$.
By flat base change, $f^*\RDERF^if_* k=H^i(G,\Orb_{G})$, which is coherent of rank
$d_i=\binom{g-n}{i}$. By the observation above, $\RDERF^if_*k$ is a trivial vector
bundle of the same rank.

Consider the Leray spectral sequence:
\[
E_2^{pq}=H^p(BG,\RDERF^qf_*k) \Rightarrow E_\infty^{p+q}=H^{p+q}(\spec k,k).
\]
Of course, $H^n(\spec k,k)=0$, unless $n=0$. Since $\RDERF^qf_*k$ is trivial, we also
have that $E_2^{pq}=H^p(BG,\Orb_{BG})\otimes_k k^{d_q}$.

If $n=g$, then $E_2^{pq}=0$ for all $q>0$, so the spectral sequences
degenerates and we deduce that $H^p(BG,\Orb_{BG})=0$ if $p>0$. It follows that
$BG$ has cohomological dimension zero.

If $n<g$, then we claim that $BG$ does not have finite cohomological dimension.
In fact, suppose on the contrary that $BG$ has finite cohomological
dimension. Then $E_2$ is bounded with Euler characteristic zero, since
$\sum_{i=0}^{g-n} (-1)^i d_i=0$. This gives a contradiction since the Euler
characteristic of $E_\infty$ is one.
\end{proof}
\begin{remark}\label{rem:strange_cd}
The groups $G=S\times_A E(A)$ have quite curious properties. The classifying
stack $BG$ has cohomological dimension zero although $G$ is not linearly
reductive (for which we require $G$ affine), showing that
\itemref{TI:Nagata:coh-dim-0} does not always imply \itemref{TI:Nagata:lin-red}
in Theorem~\ref{T:Nagata}. Moreover, the presentation $f\colon \spec k\to BG$
has cohomological dimension zero although $f$ is not affine. This shows that
in~\cite[Lem.~2.2 (6)]{perfect_complexes_stacks}, the assumption that $Y$
has quasi-affine diagonal cannot be weakened beyond affine stabilizers. We also obtain an example of an extension
$0\to U\to E(A)\to A\to 0$ such that $\cohdim(BU)=g$,
$\cohdim(BE(A))=0$ and $\cohdim(BA)=\infty$ for every $g\geq 1$. This shows
that in Lemma~\ref{lem:jarod}, the cohomological dimension of $BQ$ is not
bounded by those of $BG$ and $BH$ unless $\cohdim(BH)=0$.
\end{remark}

\begin{remark}
In the proof of Proposition~\ref{prop:coh-dim-of-anti-affine}, we did not
calculate the cohomology of $BG$ for an anti-affine group scheme $G$. This can
be done in characteristic zero as follows. Recall that $G$ is the extension of
the abelian variety $A$ of dimension $g$ by a commutative group
$G_\aff=T\times U$, where $T$ is a torus and $U\cong (\Ga)^n$ is a
unipotent group of dimension $0\leq n\leq g$. As before, we let $W\subseteq
H^1(A,\Orb_A)^\vee$ be the $k$-vector space (of dimension $g-n$)
corresponding to the vector extension $0\to U\to E\to A\to 0$. Then,
\[
H^j(BG,\Orb_{BG})=H^j(BE,\Orb_{BE})=\begin{cases}
\mathrm{Sym}^d(W^\vee) & \text{if $j=2d\geq 0$,}\\
0 & \text{otherwise.}
\end{cases}
\]
The first equality holds since $BT$ has cohomological dimension zero. The
second equality follows by induction on $g-n$. When $g-n=0$ we saw that
there is no higher cohomology. For $g-n>0$, we consider the Leray spectral
sequence for $BE'\to BE\to \spec k$ where $E'$ is a vector extension of $A$
corresponding to a subspace $W'\subseteq W$ of dimension $g-n-1$. An easy
calculation gives the desired result.

In positive characteristic, $n=0$ and $E=A$ and we expect that the cohomology
is the same as above (with $W=H^1(A,\Orb_A)^\vee$). When $g=1$, that is, when
$A$ is an elliptic curve, the Leray spectral sequence for $\spec k\to BA\to
\spec k$ and an identical calculation as above confirms this.
\end{remark}
\section{Stabilizer groups and cohomological dimension}\label{S:stab-coh-dim}
In this section, we generalize a result of Gaitsgory and
Drinfeld~\cite[Thm.~1.4.2]{MR3037900} on the cohomological dimension of noetherian algebraic stacks in characteristic zero with
affine stabilizers.
We extend their result to positive characteristic 
and also allow stacks with non-finitely presented inertia. 



%
\begin{theorem}\label{T:concentrated-stacks}
Let $X$ be a quasi-compact and quasi-separated algebraic stack with affine stabilizers. If $X$ is either
\begin{enumerate}
\item a $\Q$-stack, or
\item has nice stabilizers, or
\item has nice stabilizers at points of positive
  characteristic and finitely presented inertia,
\end{enumerate}
then $X$ is concentrated. In particular, this is the case if $X$ is a tame
Deligne--Mumford stack, or a tame Artin stack~\cite{MR2427954}. 
\end{theorem}


Note that Theorems~\ref{T:Nagata} and~\ref{T:conc-BG-char-zero} give a partial
converse to Theorem \ref{T:concentrated-stacks}: if $X$ is concentrated, then the stabilizer groups of $X$ are either
\begin{enumerate}
\item of positive characteristic and nice;
\item of characteristic zero and affine; or
\item of characteristic zero and extensions of an affine group by
  an anti-affine group of the form $S\times_A E(A)$.
\end{enumerate}
Theorem~\ref{MT:concentrated-stacks} follows from
Theorem~\ref{T:concentrated-stacks} and this converse.

We will prove Theorem~\ref{T:concentrated-stacks} by stratifying the stack
into pieces that admit easy descriptions. For nice stabilizers, we need the
following
\begin{definition}
A morphism of algebraic stacks $X\to Y$ is \emph{nicely presented} if
there exists:
\begin{enumerate}
\item a constant finite group $H$ such that $|H|$ is invertible over $X$,
\item an $H$-torsor $E\to X$, and
\item a $(\Gm)^n$-torsor $T\to E$ such that $T\to Y$ is quasi-affine.
\end{enumerate}
We say that $X\to Y$ is \emph{locally nicely
presented} if $X\times_Y Y'\to Y'$ is nicely presented for some fppf-covering
$Y'\to Y$.
\end{definition}

Note that a locally nicely presented morphism has finite cohomological
dimension.
If $Y$ has nice stabilizers (e.g., $Y$ is a scheme) and $X\to Y$
is locally nicely presented, then $X$ has nice stabilizers. The following lemma will also be useful.
\begin{lemma}\label{L:nice_cons}
  Let $G$ be a group algebraic space of finite presentation over a
  scheme~$S$. If $G$ has affine fibers, then the locus in $S$ where
  the fibers are nice group schemes is constructible.
\end{lemma}
\begin{proof}
  Standard arguments reduce the situation to the following: $S$ is noetherian and integral with
  generic point~$s$ and $G$ is affine and flat over $S$. We may also
  replace $S$ with $S'$ for any dominant morphism $S'\to S$
  of finite type. In particular, we may replace the residue field of
  the generic point with a finite field extension. Note that if the
  generic point has characteristic $p$, then $S$ is an $\F_p$-scheme.

  If the connected component of $G_s$ is not of multiplicative type, then there
  exists, after a finite field extension, either a subgroup $\Ga \to G_s$ or a
  subgroup $\Galpha_p\to G_s$. By smearing out, there is an induced closed
  subgroup $(\Ga)_U \to G_U$ or $(\Galpha_p)_U\to G_U$, where $U$ is open and
  dense in $S$; in particular, $G_u$ is not nice for every $u\in U$.

  If the connected component of $G_s$ is of multiplicative type, there is,
  after a residue field extension, a sequence $0 \to T_s \to G_s \to H_s \to 0$
  with $T_s$ diagonalizable and $H_s$ constant. We have $T$ and $H$ over $S$
  and we can spread out to an exact sequence over an open dense subscheme $U$ of
  $S$ that agrees with the pull back of $G$ to $U$.

  Let $d$ be the order of $H_s$ and $p$ the characteristic of $\kappa(s)$.
  If $G_s$ is nice, then $p\nmid d$. If $p$ is zero, we may shrink
  $U$ such that no point has characteristic dividing~$d$. Thus $G_u$ is nice for every $u\in U$.
  Conversely, if $G_s$ is not nice, then $p \mid d$ and $G_u$ is not nice for every $u\in U$.
\end{proof}

\begin{definition}
Let $X$ be an algebraic stack. A \emph{finitely presented filtration}
$(X_i)_{i=0}^r$ is a sequence of finitely presented closed substacks
$\emptyset=X_0\hookrightarrow X_1\hookrightarrow \dots \hookrightarrow
X_r\hookrightarrow X$ such that $|X_r|=|X|$.
\end{definition}
\begin{remark}\label{rem:strata_gerbes}
  If $X$ is quasi-compact and quasi-separated with inertia of finite
  presentation (e.g., $X$
  noetherian), then there exists a finitely presented filtration of
  $X$ with strata that are gerbes. In the noetherian case, this is immediate from
  generic flatness and \cite[Prop.~10.8]{MR1771927}. For the general
  case, see \cite[Cor.~8.4]{rydh-2014}. Moreover, by Lemma
  \ref{L:nice_cons}, if $X$ has affine stabilizers as in Theorem
  \ref{MT:fin-coh-dim-of-BG}(1) or (2), then $X$ has a stratification
  by gerbes such that each stratum is either of equal characteristic $0$ or
  nice.
\end{remark}
On a quasi-compact and quasi-separated algebraic stack, every
quasi-coherent sheaf is a direct limit of its finitely generated
quasi-coherent subsheaves. This is well-known for noetherian algebraic
stacks \cite[Prop.~15.4]{MR1771927}. The general case was recently
settled by the second author \cite{rydh-2014}.

\begin{proposition}\label{P:stabilizer-filtration}
Let $X$ be a quasi-compact and quasi-separated algebraic stack. Then
\begin{enumerate}
\item\label{PI:stabfiltr:affine}
  $X$ has \emph{affine} stabilizers if and only if there exists a finitely
  presented filtration $(X_i)_{i=0}^r$, positive integers $n_1,n_2,\dots,n_r$
  and \emph{quasi-affine} morphisms $X_i\setminus X_{i-1}\to
  B\GL_{n_i,\Z}$ for every $i=1,\dots,r$; and
\item\label{PI:stabfiltr:nice}
  $X$ has \emph{nice} stabilizers if and only if there exists a finitely
  presented filtration $(X_i)_{i=0}^r$, affine schemes $S_i$ of finite
  presentation over $\spec \Z$
  and \emph{locally nicely presented} morphisms $X_i\setminus X_{i-1}\to S_i$
  for every $i=1,\dots,r$.
\end{enumerate}
\end{proposition}
\begin{proof}
The conditions are clearly sufficient. To prove that they are necessary,
first assume that $X$ is an fppf-gerbe over an algebraically closed field $k$.
Then $X=BG$, where $G$ is an affine (resp.\ nice) group
scheme. If $G$ is affine, then
there is a quasi-affine morphism to some $B\GL_{n,k}$~\cite[Lem.~3.1]{MR2108211}.
If $G$ is nice, then $BG^0\to BG$ is an $H=\pi_0(G)$-torsor.
Since $G^0$ is diagonalizable, there is a $(\Gm)^n$-torsor
$(\mathbb{G}_{m,k})^{n-r}\to BG^0$.
Thus $BG\to \spec k\to \spec \Z$ is nicely presented. 

If $k$ is not algebraically closed, then, by approximation, the situation above
holds after passing to a finite field extension $k'/k$. If the stabilizer of
$X$ is affine, then $X$ has the resolution
property~\cite[Rmk.~7.2]{perfect_complexes_stacks} and hence there is a
quasi-affine
morphism
$X\to B\GL_{n,k}$. In this case, let $S=B\GL_{n,\Z}$.
If the stabilizer of $X$ is nice, then $X\to \spec k$ is at least locally
nicely presented. By approximating $\spec k\to \spec \Z$, we obtain a
finitely presented affine scheme $S\to \spec \Z$
such that $X\to \spec k\to S$ is locally nicely presented.

If $X$ is any quasi-separated algebraic stack, then for every point $x\in |X|$
there is an immersion $Z\hookrightarrow X$ such that $Z$ is an fppf-gerbe over an affine
integral scheme $\underline{Z}$ and the residual gerbe $\mathcal{G}_x\to \spec
\kappa(x)$ is the generic fiber of $Z\to \underline{Z}$~\cite[Thm.~B.2]{MR2774654}.
In particular, $\mathcal{G}_x$ is the inverse limit of open neighborhoods
$U\subseteq Z$ of $x$ such that $U\to Z$ is affine.
By~\cite[Thm.~C]{rydh-2009}, there exists
an open neighborhood $x\in U\subseteq Z$ and a morphism $U\to S$ that is
quasi-affine
(resp.\ locally nicely presented).

We may write the quasi-compact immersion $U\hookrightarrow Z\hookrightarrow X$
as a closed immersion $U\hookrightarrow V$ in some quasi-compact open substack
$V\subset X$. Since $V$ is quasi-compact and quasi-separated, we may express 
$U\hookrightarrow V$ as an inverse limit of finitely presented closed
immersions $U_\lambda\hookrightarrow V$. Since $S$ is of finite presentation,
there is a morphism $U_\lambda\to S$ for sufficiently large $\lambda$.
After increasing $\lambda$, the morphism $U_\lambda\to S$ becomes quasi-affine
(resp.\ locally nicely presented) by~\cite[Thm.~C]{rydh-2009}.
Let $U_x=U_\lambda$.

For every $x\in |X|$ proceed as above and choose a locally closed finitely
presented immersion $U_x\hookrightarrow X$ with $x\in |U_x|$. As the substacks
$U_x$ are constructible, it follows by quasi-compactness that a finite number
of the $U_x$'s cover $X$ and we easily obtain a stratification and filtration
as claimed, cf.~\cite[Pf.\ of Prop.~4.4]{MR2774654}.
\end{proof}
The following lemma will be useful.
\begin{lemma}[{\cite[2.3.2]{MR3037900}}]\label{L:stratify_cohdim}
Let $X$ be a quasi-compact and quasi-separated algebraic stack. If $i\colon Z\hookrightarrow X$ is a
finitely presented closed immersion with complement
$j\colon U\hookrightarrow X$, then
\[
\cohdim(X)\leq \max\{\cohdim(U),\cohdim(Z)+\cohdim(j)+1\}.
\]
\end{lemma}
\begin{proof}
Let $I$ denote the ideal sheaf defining $Z$ in $X$. Let $F$ be a
quasi-coherent sheaf on $X$. Consider the adjunction map $F\to \RDERF j_*j^*F$
and let
$C$ denote the cone. Then $j^*C=0$ and $C$ is supported in degrees
$\leq \cohdim(j)$. Since $H^d(\RDERF\Gamma \RDERF j_*j^*F)=H^d(U,j^*F)=0$ for
$d> \cohdim(U)$, it
is enough to show that $H^d(X,G)=0$ if $G$ is a quasi-coherent sheaf such
that
$j^*G=0$ and $d>\cohdim(Z)$. After writing $G$ as a direct limit of its finitely
generated
subsheaves, we may further assume that $G$ is finitely generated. Then
$I^nG=0$ for sufficiently large $n$ and one easily proves that $H^d(X,G)=0$
by induction on $n$.
\end{proof}
We now prove the main result of this section.
\begin{proof}[Proof of Theorem~\ref{T:concentrated-stacks}]
We first treat (1) and (2). Choose a filtration as in
Proposition~\ref{P:stabilizer-filtration}\itemref{PI:stabfiltr:affine} or
\itemref{PI:stabfiltr:nice}.
In characteristic zero, $B\GL_n$ has cohomological dimension zero and
quasi-affine morphisms have finite cohomological dimension. In arbitrary
characteristic, locally nicely presented morphisms have finite cohomological
dimension. Indeed, $BH$ and $B(\Gm)^n$ have cohomological dimension zero.
Thus, the Theorem follows from Lemma \ref{L:stratify_cohdim}. For (3), we may choose a filtration as in Remark \ref{rem:strata_gerbes}. Then the result follows from Lemma \ref{L:stratify_cohdim} and the cases (1) and (2) already proved. 
\end{proof}

There are several other applications of the structure result of
Proposition~\ref{P:stabilizer-filtration}. An immediate corollary is that the
locus of points where the stabilizers are affine
(resp.\ nice) is ind-constructible. This is false for ``linearly reductive'':
the locus with linearly reductive stabilizers in $B\GL_{n,\Z}$, for $n\geq 2$,
is the subset $B\GL_{n,\Q}$ which is not ind-constructible.
Another corollary is the following
approximation result.

\begin{theorem}
Let $S$ be a quasi-compact algebraic stack and let $X=\varprojlim_\lambda
X_\lambda$ be an inverse limit of quasi-compact and quasi-separated morphisms
of algebraic stacks $X_\lambda\to S$
with affine transition maps. Then $X$ has affine
(resp.\ nice) stabilizers if and only if $X_\lambda$ has
affine (resp.\ nice) stabilizers for sufficiently large
$\lambda$.
\end{theorem}
\begin{proof}
The question is fppf-local on $S$, so we can assume that $S$ is affine.
Note that if $X\to Y$ is affine and
$Y$ has affine (resp.\ nice) stabilizers, then so has
$X$. The result now follows from Proposition~\ref{P:stabilizer-filtration}
and~\cite[Thm.~C]{rydh-2009}.
\end{proof}

Thus if $X_\lambda$ is of equal characteristic and has affine stabilizer
groups, then $X\to S$ has finite cohomological
dimension if and only if $X_\lambda\to S$ has finite cohomological dimension
for sufficiently large $\lambda$. The example $X=B\GL_{2,\Q}=\varprojlim_m B\GL_{2,\Z[\frac{1}{m}]}$ shows that this is false in mixed characteristic.

\section{Compact generation of classifying stacks}\label{S:compact-generation}
In this section, we prove Theorem~\ref{MT:compact-gen-of-BG} on the compact
generation of classifying stacks. The following three lemmas will be useful.
\begin{lemma}\label{L:cons-adjoint}
  Let $F \colon \mathcal{T} \to \mathcal{S}$ be a triangulated functor
  between triangulated categories that are closed under small
  coproducts. Assume that $F$ admits a conservative right adjoint $G$ that 
  preserves small coproducts. If $\mathcal{T}$ is
  compactly generated by a set $T$, then $\mathcal{S}$ is compactly
  generated by the set $F(T)=\{F(t) \suchthat t \in T\}$.
\end{lemma}
\begin{proof}
  By \cite[Thm.~5.1~``$\Rightarrow$'']{MR1308405}, $F(T) \subseteq
  \mathcal{S}^c$. Thus, it remains to prove that the set $F(T)$ is
  generating. If $s \in \mathcal{S}$ is non-zero, then $G(s)$ is
  non-zero. It follows that there is a non-zero map $t \to G(s)[n]$
  for some $t\in T$ and $n\in \Z$. By adjunction, there is a non-zero
  map $F(t) \to s[n]$, and we have the claim.
\end{proof}
\begin{lemma}\label{L:qppsh}
  Let $\pi\colon X' \to X$ be a proper and faithfully flat morphism of
  noetherian algebraic stacks. Assume that either $\pi$ is finite or a
  torsor for a smooth group scheme. If a set $T$ compactly generates
  $\DQCOH(X')$, then the set $\{\RDERF \pi_*P \suchthat P \in T\}$
  compactly generates $\DQCOH(X)$.
\end{lemma}
\begin{proof}
  By \cite[Ex.~6.5]{perfect_complexes_stacks} and Proposition
  \ref{P:quasi-perfect-smooth-proper}, in both cases $\RDERF\pi_*$ is
  $\DQCOH$-quasiperfect with respect to open immersions
  (see \cite[Defn.~6.4]{perfect_complexes_stacks}) and
  its right adjoint $\pi^!$ is conservative. The claim now
  follows from Lemma \ref{L:cons-adjoint}.
\end{proof}
\begin{lemma}\label{L:cons-extension}
  Let $k$ be a field and let $1 \to K \to G \to H \to 1$ be a short
  exact sequence of group schemes of finite type over $k$.
  Let $p\colon BG \to BH$ be
  the induced morphism. Assume that either
  \begin{enumerate}
  \item \label{lem:conservative:trivgen} $\DQCOH(BK)$ is compactly
    generated by $\Orb_{BK}$, or
  \item \label{lem:conservative:aa} $K \subseteq G_\ant$ and
    $\cohdim(BK) = 0$.
  \end{enumerate}
  Then $\RDERF p_* \colon \DQCOH(BG) \to \DQCOH(BH)$ is concentrated
  and conservative.
\end{lemma}
\begin{proof}
  For \itemref{lem:conservative:trivgen}, the pull back of $p$ along the universal $H$-torsor is the
  morphism $p'\colon BK \to \spec k$. Since $\DQCOH(BK)$ is compactly
  generated by $\Orb_{BK}$, it follows that $BK$ is concentrated. By
  \cite[Lem.~2.5(2)]{perfect_complexes_stacks}, $p$ is
  concentrated. To prove that $\RDERF p_*$ is conservative, by
  \cite[Thm.~2.6]{perfect_complexes_stacks}, it remains to prove that
  $\RDERF p'_*$ is conservative. If $M\in \DQCOH(BK)$ is non-zero,
  then by assumption there is a non-zero map $\Orb_{BK}[n] \to M$ for
  some integer $n$. Since $\LDERF p'^*\Orb_{\spec k} \homotopic
  \Orb_{BK}$, by adjunction, there is a non-zero map $\Orb_{\spec
    k}[n] \to \RDERF p'_*M$. The claim follows.

  For \itemref{lem:conservative:aa}, by
  \cite[Lem.~2.2(2)~\&~2.5(2)]{perfect_complexes_stacks}, $\cohdim(p) = 0$ and
  $p$ is concentrated. Thus, if $M \in \DQCOH(BG)$ and $i\in \Z$, then
  $\COHO{i}(\RDERF p_*M) = p_*\COHO{i}(M)$. So to establish that
  $\RDERF p_*$ is conservative, it remains to prove that the functor
  $p_*\colon \QCOH(BG) \to \QCOH(BH)$ is conservative. Let $q\colon
  BG \to B(G/G_\ant)$ be the natural morphism. Then $q$
  factors as $BG \xrightarrow{p} BH \to B(G/G_\ant)$. Smooth-locally $q$ is the morphism $BG_\ant
  \to \spec k$, and $\QCOH(BG_\ant) \to \QCOH(\spec k)$ is an
  equivalence \cite[Lem.~1.1]{MR2488561}. By descent, it follows that $q_*$ is conservative. Hence, $p_*$ is conservative. The result follows.
\end{proof}
\begin{proof}[Proof of Theorem~\ref{MT:compact-gen-of-BG}]
If $k$ has positive characteristic and $\red{\overline{G}}^0$ is not
semi-abelian, then $B_kG$ is poorly
stabilized~\cite[Lem.~4.1]{hallj_neeman_dary_no_compacts}, so
$\DQCOH(B_kG)$ is not compactly
generated~\cite[Thm.~1.1]{hallj_neeman_dary_no_compacts}. Conversely, assume
either that $k$ has characteristic zero or that $\red{\overline{G}}^0$ is
semi-abelian.

Let $G^0$ be the connected component of $G$. Then $BG^0\to BG$ is
finite and faithfully flat. By Lemma \ref{L:qppsh}, we may assume that
$G=G_0$. By Lemma \ref{L:qppsh}, we may always pass to finite extensions of the ground field $k$. In particular, we may assume that $\red{G}$ is a
smooth group scheme. Similarly, since $B\red{G}\to BG$ is finite and
faithfully flat, we may replace $G$ with $\red{G}$. Hence, we may assume
that $G$ is smooth and connected.

By Chevalley's Theorem \cite[Thm.~1.1]{MR1906417}, we may (after
passing to a finite extension of $k$) write $G$ as an extension of an
abelian variety $A$ by a smooth connected affine group $G_\aff$. By assumption,
$G_\aff$ is a torus in positive characteristic. In particular, $BG_\aff$ is
concentrated, has affine diagonal and the resolution property; thus
$\DQCOH(BG_\aff)$ is compactly generated by a set of compact vector bundles
\cite[Prop.~8.4]{perfect_complexes_stacks}. Since the induced map $f\colon
BG_\aff \to BG$ is an $A$-torsor, $\DQCOH(BG)$ is compactly generated (Lemma \ref{L:qppsh}). Note that this also establishes
\itemref{MT:compact-gen-of-BG:comp-reps}.

For \itemref{MT:compact-gen-of-BG:irreps-fcd}, let $M \in \DQCOH(B_kG)$ and suppose that $M\not\simeq 0$. By \itemref{MT:compact-gen-of-BG:comp-reps}, there exists a non-zero map $V[n] \to M$, where $V$ is a finite-dimensional $k$-representation of $G$. Let $L \subseteq V$ be an irreducible $k$-subrepresentation of $G$. If the composition $L[n] \to V[n] \to M$ is zero, then there is an induced non-zero map $(V/L)[n] \to M$. Since $V$ is finite-dimensional, we must eventually arrive at the situation where there is a non-zero map $L[n] \to M$, where $L$ is irreducible. Finally, $B_kG$ has finite cohomological dimension (Theorem \ref{MT:fin-coh-dim-of-BG}), so $L$ is compact \cite[Rem.~4.6]{perfect_complexes_stacks}.

It remains to address \itemref{MT:compact-gen-of-BG:one}. Suppose that
$\DQCOH(BG)$ is compactly generated by a single perfect complex. Then
so too is $\DQCOH(B\red{\bar{G}}^0)$. Assume that
$k=\bar{k}$
and $G=\red{\bar{G}}^0$; in particular, $G$ is smooth and connected
and $k$ is perfect. To derive a contradiction, we assume
that $G/G_\ant$---the affinization of $G$---is not unipotent.
By Chevalley's Theorem \cite[Thm.~1.1]{MR1906417},
$G$ is an extension of an abelian variety $A$ by a connected smooth affine
group $G_\aff$. The exact sequence of \cite[Prop.~3.1(i)]{MR2488561}
quickly implies that the induced map $G_\aff \to G/G_\ant$ is
surjective. In particular, $G_\aff$ is not unipotent;
moreover, there is a subgroup
$\Gm \subset G_\aff$ such that the induced map $\Gm \to G/G_\ant$ has kernel $\mu_n$ for some $n$. Since $G/G_\ant$ is affine and $\Gm$ is linearly reductive, it follows that the induced morphism $\phi \colon B(\Gm/\mu_n) \to B(G/G_\ant)$ is affine; in particular, the functor $\RDERF\phi_*$ is conservative. 

Let $\mathcal{L}$ be the standard representation of $\Gm$. Then for
every integer $r$, a brief calculation using that $\RDERF \phi_*$ is
conservative proves that $\RDERF q_*(\mathcal{L}^{\tensor rn}) \neq
0$, where $q$ is the composition $B\Gm \to B(\Gm/\mu_n)
\xrightarrow{\phi} B(G/G_\ant)$. If $\DQCOH(BG)$ is
compactly generated by a single perfect complex $P$, then for every
integer $r$ there exist integers $m_r$ and non-zero maps $l_r \colon P
\to \RDERF \psi_*(\mathcal{L}^{\tensor rn})[m_r]$, where $\psi \colon
B\Gm \to BG_\aff\to BG$ is the induced map; indeed, $\RDERF q_*$ is conservative
so $\RDERF \psi_*(\mathcal{L}^{\tensor rn}) \neq 0$ for every $r$. By
adjunction, there are non-zero maps $\LDERF \psi^*P \to
\mathcal{L}^{\tensor rn}[m_r]$, for every $r$. That is,
\[
\Hom_{\Orb_{B\Gm}}(\LDERF \psi^*P, \mathcal{L}^{\tensor rn}[m_r]) = \Hom_{\Orb_{B\Gm}}(\psi^*\COHO{m_r}(P),\mathcal{L}^{\tensor rn})
\]
is non-zero for every integer $r$. But $\LDERF \psi^*P$ is perfect, so there are only finitely many non-zero $\COHO{i}(P)$ and only a finite number of the representations $\mathcal{L}^{\tensor rn}$ appear in $\psi^*\COHO{i}(P)$. Hence, we have a contradiction, so the affinization of $\red{\bar{G}}^0$ is unipotent.

Conversely, suppose that the affinization of $\red{\bar{G}}^0$ is
unipotent. By Lemma \ref{L:qppsh} and arguing as before,
after passing to a finite extension of $k$, we may assume that
$G=\red{G}^0$ and that the affinization $G/G_\ant$ is unipotent.
Passing to
a further finite extension of $k$, by Chevalley's Theorem
\cite[Thm.~1.1]{MR1906417}, we may assume that $G$
(resp.\ $G_\ant$) is an extension of an abelian scheme $A$
(resp.\ $A'$) by a connected smooth affine group $G_\aff$ (resp.\ $G_\aff'$). Note that
if $k$ has positive characteristic, then since $\DQCOH(BG)$ is
compactly generated, it follows by what we have already established
that $G_\aff$ has no unipotent elements; in particular, since $G_\aff \to
G/G_\ant$ is surjective (arguing as above) and $G/G_\ant$
is unipotent, it follows that $G/G_\ant$ is trivial.

By \cite[Prop.~3.1(ii)]{MR2488561} we have that $G_\aff' \subseteq G_\aff$. Since $G_\aff'$ is
smooth, affine, connected and commutative, it follows that $G_\aff'=T\times
U$, where $T$ is a torus and $U$ is connected and unipotent
\cite[(2.5)]{MR2488561}. Note that from the above, if $k$ has positive
characteristic, then $G_\aff'=T$. By assumption, $G$ is connected; thus,
$G_\ant \subseteq Z(G)$ \cite[Cor.~III.3.8.3]{MR0302656}. In
particular, $T$ is a normal subgroup of both $G_\aff$ and $G$. By Lemmas
\ref{L:cons-adjoint} and \ref{L:cons-extension}\itemref{lem:conservative:aa},
it suffices to prove that $\DQCOH(G/T)$ is compactly generated by a
single perfect complex.

We have exact sequences
\[
\xymatrix@R-1pc{1 \ar[r] & G_\aff/T \ar[r] & G/T \ar[r] & A \ar[r] & 1\phantom{.} \\ 
1 \ar[r] &  U \ar[r] & G_\aff/T \ar[r] & G_\aff/G_\aff' \ar[r] & 1. }
\]
The kernel of the surjective map $G_\aff/G_\aff'\twoheadrightarrow G/G_\ant$ is
finite by~\cite[Prop.~3.1(ii)]{MR2488561}. By assumption $G/G_\ant$
is unipotent and $G_\aff/G_\aff'$ is connected and smooth; hence $G_\aff/G_\aff'$ and $G_\aff/T$ are
unipotent. Note that in positive characteristic $G_\aff/T=0$.

We know that $\DQCOH(BA)$ is compactly generated by a single perfect complex
(Lemma~\ref{L:qppsh}). In characteristic zero, since $G_\aff/T$ is unipotent, we
have also established that $\DQCOH(B(G_\aff/T))$ is compactly generated by the
structure sheaf in~\itemref{MT:compact-gen-of-BG:irreps-fcd}. Hence, by Lemmas
\ref{L:cons-adjoint} and
\ref{L:cons-extension}\itemref{lem:conservative:trivgen}, we have that
$\DQCOH(B(G/T))$ is compactly generated and the result follows.
\end{proof}
\begin{remark}\label{rem:gens-of-BG} 
In characteristic zero, the proof of Theorem \ref{MT:compact-gen-of-BG} shows that if $G^0$ fits in an exact sequence of group schemes $0 \to U \to G^0 \to A \to 0$, where $U$ is unipotent, then $\DQCOH(BG)$ is compactly generated by the perfect complex $\RDERF \pi_*\Orb_{B U}$, where $\pi \colon BU \to BG$ is the induced morphism. 
\end{remark}
\begin{corollary}\label{C:gerbes}
  Let $k$ be a field. Let $\mathcal{G}$ be a
  quasi-compact and quasi-separated fppf gerbe over $\spec k$.
  The derived category $\DQCOH(\mathcal{G})$ is compactly generated if and only if
  $\mathcal{G}$ is not poorly stabilized.
\end{corollary}
\begin{proof}
  If $\mathcal{G}$ is poorly stabilized, then $\DQCOH(\mathcal{G})$ is
  not compactly
  generated~\cite[Thm.~1.1]{hallj_neeman_dary_no_compacts}. Conversely, Lemma \ref{L:qppsh} permits us to reduce to the situation where $\mathcal{G}$ is neutral. The result now follows from Theorem~\ref{MT:compact-gen-of-BG}.
\end{proof}

More generally, we have the following.
\begin{theorem}
Let $S$ be a scheme and let $G\to S$ be a flat group scheme of finite
presentation. Let $X$ be a quasi-compact algebraic stack over $S$ with
quasi-finite and separated diagonal and let $\mathcal{G}\to X$ be a
$G$-gerbe. Assume that either
\begin{enumerate}
\item $S$ is the spectrum of a field $k$ and $G$ is not poor, that is, either
$S$ has characteristic zero or $\red{\overline{G}}^0$ is semi-abelian; or
\item $S$ is arbitrary and $G\to S$ is of multiplicative type.
\end{enumerate}
Then $\mathcal{G}$ is $\aleph_0$-crisp (and $1$-crisp if $G \to S$ is
proper). In
particular, $\DQCOH(\mathcal{G})$ is compactly generated.
\end{theorem}
\begin{proof}
The question is local on $X$ with respect to quasi-finite faithfully flat
morphisms of finite presentation~\cite[Thm.~C]{perfect_complexes_stacks}. We
may thus assume that $X$ is affine and that $\mathcal{G}\to X$ is a trivial
$G$-gerbe, that is, $\mathcal{G}\simeq X\times_S BG$. We may also replace $S$ by
a quasi-finite flat cover and in the first case assume that $\red{G}^0$ is
a group scheme and in the second case assume that $G\to S$ is diagonalizable.

In the second case $X\times_S BG$ is concentrated, has affine diagonal and has 
the resolution property. It is thus
$\aleph_0$-crisp~\cite[Prop.~8.4]{perfect_complexes_stacks}.

In the first case, we may, after further base change, apply Chevalley's theorem
and write $\red{G}^0$ as an extension of an abelian variety $A/k$ by a smooth
connected affine group $G_\aff$ (a torus in positive characteristic). The stack
$X\times_k
BG_\aff$ is $\aleph_0$-crisp as in the previous case ($1$-crisp if $G$ is
proper). The morphism
$X\times_k BG_\aff\to X\times_k B\red{G}^0$ is a torsor under $A$, hence
Proposition~\ref{P:quasi-perfect-smooth-proper} and \cite[Prop.~6.6]{perfect_complexes_stacks} applies. Hence, $X\times_k B\red{G}^0$ is $\aleph_0$-crisp. Finally, since $B\red{G}^0\to
BG$ is finite and flat, $X\times_k BG$ is $\aleph_0$-crisp
by~\cite[Thm.~C]{perfect_complexes_stacks}.
\end{proof}
%
\appendix
\section{Grothendieck duality for smooth and
  representable morphisms of algebraic stacks} \label{A:duality}
In this Appendix we prove a variant of
\cite[Prop.~1.20]{2008arXiv0811.1955N} that was necessary for this
paper. The difficult parts of the following Proposition, for schemes, are
well-known \cite[Thm.~4.3.1]{MR1804902}.

Recall that a morphism of algebraic stacks $X\to Y$ is \emph{schematic} (or
strongly representable) if for every scheme $Y'$ and morphism $Y'\to Y$, the
pull-back $X\times_Y Y'$ is a scheme. We say that $X\to Y$ is \emph{locally
schematic} if there exists a faithfully flat morphism $Y'\to Y$, locally
of finite presentation, such that $X\times_Y Y'$ is a scheme. In particular, if $S$
is a scheme, $G\to S$ is a group \emph{scheme}, $Y$ is an $S$-stack
and $X\to Y$ is a $G$-torsor, then $X\to Y$ is locally schematic (but perhaps
not schematic).

\begin{proposition}\label{P:quasi-perfect-smooth-proper}
  Let $f\colon X \to Y$ be a proper, smooth, and locally schematic
  morphism of noetherian algebraic stacks of relative dimension
  $n$. Let $f^!\colon \DQCOH(Y) \to \DQCOH(X)$ be the functor
  $\omega_{f}[n] \otimes_{\Orb_X} \LDERF f^*(-)$, where
  $\omega_{f} = \wedge^n \Omega_{f}$.
  \begin{enumerate}
  \item There is a \emph{trace morphism} $\gamma_f \colon \RDERF^n f_*\omega_f
    \to \Orb_Y$ that is compatible with locally noetherian base change
    on $Y$.
  \item The trace morphism induces a natural transformation
    $\Tr_f\colon \RDERF f_*f^! \to \ID{}$, which is compatible with
    locally noetherian base change and gives rise to a
    sheafified duality quasi-isomorphism whenever $M\in \DQCOH(X)$ and
    $N\in \DQCOH(Y)$:
    \[
    J_{f,M,N} \colon \RDERF f_*\SRHom_{\Orb_X}(M,f^!N) \to
    \SRHom_{\Orb_Y}(\RDERF f_*M,N).
    \]
    In particular, $f^!$ is a right adjoint to
    $\RDERF f_*\colon \DQCOH(X)\to \DQCOH(Y)$. 
  \end{enumerate}
\end{proposition}
\begin{proof}
   For the moment, assume that $f$ is a
  morphism of schemes. By \cite[Cor.~3.6.6]{MR1804902},
  there is a \emph{trace morphism} $\gamma_{f} \colon
  \RDERF^nf_*\omega_{f} \to \Orb_Y$ that is compatible with locally
  noetherian base change on $Y$. For $N\in \DQCOH(Y)$, there is also
  an induced morphism, which we denote as $\Tr_f(N)$:
  \[
  \RDERF f_* f^!N\simeq (\RDERF f_* \omega_{f})[n]
  \otimes^{\LDERF}_{\Orb_{Y}} N \to (\RDERF^n f_*\omega_{f})
  \otimes^{\LDERF}_{\Orb_Y} N \xrightarrow{\gamma_f \tensor \ID{}} N,   
  \]
  where the first isomorphism is the Projection Formula
  \cite[Prop.~5.3]{MR1308405} and the second morphism is given by the
  truncation map $\tau_{\geq 0}$---using that $\RDERF f_*$ has cohomological
  dimension $n$. Tor-independent base change (e.g., \cite[Cor.~4.13]{perfect_complexes_stacks}) shows that the morphism $\Tr_f(N)$ is natural and
  compatible with locally noetherian base change and induces a sheafified duality
  morphism:
  \[
  J_{f,M,N} \colon \RDERF f_* \SRHom_{\Orb_X}(M,f^!N) \to
  \SRHom_{\Orb_Y}(\RDERF f_*M,N),
  \]
  where $M \in \DQCOH(X)$ and $N \in \DQCOH(Y)$. The morphism
  $J_{f,M,N}$ is a quasi-isomorphism whenever $M \in
  \DCAT_{\COH}^b(X)$ and $N \in \DQCOH^b(Y)$ \cite[Thm.~4.3.1]{MR1804902}.

  Returning to the general case, we note that by hypothesis, there is
  a noetherian scheme $U$ and a smooth and surjective
  morphism $p\colon U \to Y$ such that in the $2$-cartesian square of
  algebraic stacks:
  \[
  \xymatrix{X_U \ar[d]_{f_U}\ar[r]^{p_X} & X \ar[d]^f \\ U\ar[r]^p & Y,}
  \]
  the morphism $f_U$ is a proper and smooth morphism of relative
  dimension $n$ of noetherian schemes.  Let $R=U\times_Y U$, which is
  a noetherian algebraic space. Let $\tilde{R} \to R$ be an \'etale
  surjection, where $\tilde{R}$ is a noetherian scheme. Let $s_1$ and
  $s_2$ denote the two morphisms $\tilde{R} \to R \to U$ and let
  $f_{\tilde{R}} \colon X_{\tilde{R}} \to \tilde{R}$ denote the
  pullback of $f$ along $p\circ s_1\colon \tilde{R} \to Y$. By the above,
  there are trace morphisms $\gamma_{f_U}$ and
  $\gamma_{f_{\tilde{R}}}$ that are compatible with locally noetherian
  base change. In particular, for $i=1$ and $i=2$ the
  following diagram commutes:
  \[
  \xymatrix{s_i^*\RDERF^n(f_U)_*\omega_{f_U} \ar[r]^-{\sim}
    \ar[d]_{s_i^*\gamma_{f_U}} & \RDERF^n(f_{\tilde{R}})_*\omega_{f_{\tilde{R}}}
    \ar[d]^{\gamma_{f_{\tilde{R}}}} \\ s_i^*\Orb_U \ar[r]^-{\sim} &
    \Orb_{\tilde{R}} }
  \]
  By smooth descent, there is a uniquely induced morphism $\gamma_f \colon
  \RDERF^nf_*\omega_f \to \Orb_Y$ such that the following diagram
  commutes:
  \[
  \xymatrix{p^*\RDERF^nf_*\omega_f \ar[r]^-{\sim}
    \ar[d]_{p^*\gamma_f}& \RDERF^n(f_U)_*\omega_{f_U}
    \ar[d]^{\gamma_{f_U}} \\ p^*\Orb_Y \ar[r]^-{\sim} & \Orb_U.}
  \]
  Now the morphism $f$ is quasi-compact,
  quasi-separated, and representable---whence concentrated
  \cite[Lem.~2.5 (3)]{perfect_complexes_stacks}. By the Projection Formula
  \cite[Cor.~4.12]{perfect_complexes_stacks}, there is a natural
  quasi-isomorphism for each $N \in \DQCOH(Y)$:
  \[
  (\RDERF f_*\omega_{f}[n]) \otimes^{\LDERF}_{\Orb_{Y}} N \simeq
  \RDERF f_*f^!N.
  \]
  Since $f$ is also proper, flat, and representable with fibers of
  relative dimension $\leq n$ it follows that $\RDERF f_*\omega_{f}
  \in \DCAT_{\COH}^{[0,n]}(Y)$. Inverting the quasi-isomorphism above and
  truncating, we obtain a natural morphism:
  \[
  \RDERF f_* f^!N \simeq (\RDERF f_* \omega_{f})[n]
  \otimes^{\LDERF}_{\Orb_{Y}} N \to (\RDERF^n f_*\omega_{f})
  \otimes^{\LDERF}_{\Orb_Y} N \xrightarrow{\gamma_f \otimes \ID{}} N,  
  \]
  which we denote as $\Tr_f(N)$. If $M \in \DQCOH(X)$ and $N\in \DQCOH(Y)$, let
  \begin{align*}
  A(M,N) &= \RDERF f_* \SRHom_{\Orb_X}(M,f^!N),\\
  B(M,N) &= 
  \SRHom_{\Orb_Y}(\RDERF f_*M,N),\text{ and}\\
  J_{f,M,N} &\colon A(M,N) \to B(M,N),
  \end{align*}
  where $J_{f,M,N}$ is the sheafified duality morphism
  induced by $\RDERF f_* f^!N\to N$.
  Furthermore, there is a natural
  isomorphism of functors:
  \[
  p^*\RDERF f_*f^! \simeq \RDERF (f_U)_*p_X^*f^! \simeq \RDERF(f_U)_*(f_U)^!p^*,
  \]
  and the following diagram is readily observed to commute for each
  $N\in \DQCOH(Y)$:
  \[
  \xymatrix@C3pc{p^*\RDERF f_* f^!N \ar[r]^-{p^*\Tr_f(N)} \ar[d] &
    p^*N \ar@{=}[d] \\ \RDERF (f_U)_*(f_U)^!p^*N \ar[r]^-{\Tr_{f_U}(p^*N)} & p^*N.  }
  \]
  We have already seen that $J_{f_U,M,N}$ is a quasi-isomorphism
  whenever $M \in \DCAT_{\COH}^b(X_U)$ and $N\in \DQCOH^b(U)$. Thus,
  by tor-independent base change and the commutativity of the diagram above, the
  morphism $J_{f,M,N}$ is a quasi-isomorphism whenever $M
  \in\DCAT_{\COH}^b(X)$ and $N\in \DQCOH^b(Y)$. 

  It remains to prove that $J_{f,M,N}$ is a quasi-isomorphism for all
  $M \in \DQCOH(X)$ and all $N \in \DQCOH(Y)$. By
  \cite[Thm.~B.1]{hallj_neeman_dary_no_compacts}, 
  $\DQCOH(X)$ and $\DQCOH(Y)$ are left-complete triangulated categories. Thus,
  we have distinguished triangles:
  \[
  N\to \prod_{k\leq 0} \tau^{\geq k}N\to \prod_{k\leq 0} \tau^{\geq k}N
  \quad \text{and} \quad
  f^!N\to \prod_{k\leq 0} \tau^{\geq k}f^!N\to \prod_{k\leq 0} \tau^{\geq k}f^!N,
  \]
  where the first maps are the canonical ones and the second maps are
  $1-\textrm{shift}$. Since $f^![-n]$ is $t$-exact, we also have a distinguished
  triangle:
  \[
  f^!N\to \prod_{k\leq 0} f^!\tau^{\geq k}N\to \prod_{k\leq 0} f^!\tau^{\geq k}N.
  \]
  Hence we have
  a natural morphism of distinguished triangles:
  \[
  \xymatrix@C-1pc{A(M,N) \ar[r]
    \ar[d]_{J_{f,M,N}} & \ar[d]_{(J_{f,M,\trunc{\geq k}N})}
    \displaystyle\prod_{k\leq 0} A(M,\tau^{\geq k}N) \ar[r]
    & \ar[d]_{(J_{f,M,\trunc{\geq k}N})}
    \displaystyle\prod_{k\leq 0}
    A(M,\tau^{\geq k}N) \\
    B(M,N) \ar[r]
    & \displaystyle\prod_{k\leq 0} B(M,\tau^{\geq k}N) \ar[r]
    & \displaystyle\prod_{k\leq 0} B(M,\tau^{\geq k}N).}
  \]
  Since $f$ has cohomological dimension $\leq n$, it follows that
  there are natural quasi-isomorphisms for every pair of integers $k$ and $p$:
  \begin{equation}\label{E:bounded-eqn}
  \begin{aligned}
    \trunc{\leq p}A(M,\tau^{\geq k}N)
    &= \trunc{\leq p}\RDERF f_*\SRHom_{\Orb_X}(M,f^!\tau^{\geq k}N) \\
    &\homotopic \trunc{\leq p}\RDERF f_*\trunc{\leq p}\SRHom_{\Orb_X}(M,\tau^{\geq k-n}f^!N)\\
    &\homotopic \trunc{\leq p}\RDERF f_*\SRHom_{\Orb_X}(\tau^{\geq k-n-p}M,f^!\tau^{\geq k}N) \\
    &= \trunc{\leq p}A(\tau^{\geq k-n-p}M,\tau^{\geq k}N) \\
    \trunc{\leq p}B(M,\tau^{\geq k}N)
    &= \trunc{\leq p}\SRHom_{\Orb_Y}(\RDERF f_* M,\tau^{\geq k}N) \\
    &\homotopic
    \trunc{\leq p}\SRHom_{\Orb_Y}(\tau^{\geq k-p}\RDERF f_*M,\tau^{\geq k}N)\\
    &\homotopic 
    \trunc{\leq p}\SRHom_{\Orb_Y}(\RDERF f_*(\tau^{\geq k-n-p}M),\tau^{\geq k}N)\\
    &=\trunc{\leq p}B(\tau^{\geq k-n-p}M,\tau^{\geq k}N).
  \end{aligned}
  \end{equation}
  Thus it is enough to establish that $J_{f,M,N}$ is a
  quasi-isomorphism when $M\in \DQCOH^+(X)$ and $N
  \in \DQCOH^+(Y)$. A similar argument,
  but this time using the homotopy colimit
$\bigoplus_{k\geq 0} \tau^{\leq k}M\to \bigoplus_{k\geq 0} \tau^{\leq k}M \to M$
 (cf.,
  \cite[Lem.~4.3.2]{MR2434692}), further permits a reduction to
  the situation where $M \in \DQCOH^b(X)$ and $N\in
  \DQCOH^b(Y)$. 

  For the remainder of the proof we fix $N\in \DQCOH^b(Y)$. Let $F_N$ be
  the functor $A(-,N)$ and let $G_N$ be the
  functor $B(-,N)$, both regarded as contravariant
  triangulated functors from $\DQCOH(X)$ to $\DQCOH(Y)$. Since $N$ is bounded
  below, the functors
  $F_N$ and $G_N$ are bounded below~\eqref{E:bounded-eqn}
  and $J_{f,-,N}$ induces a
  natural transformation $F_N \to G_N$.

  Let $C \subseteq
  \QCOH(X)$ be the collection of objects of the form $\bigoplus_{i\in I}
  L_i$, where $L_i \in \COH(X)$ and $I$ is a set. Recall that $J_{f,L,N}$
  is a quasi-isomorphism whenever $L\in \COH(X)$ and $N\in \DQCOH^b(Y)$. 
  Since $F_N$ and $G_N$ both send
  coproducts to products,
  it follows that $J_{f,\oplus L_i,N}=\prod J_{f,L_i,N}$, so
  $J_{f,L,N}$ is also a quasi-isomorphism whenever $L=\bigoplus L_i \in C$.

  Every $M \in \QCOH(X)$ is a quotient of some object of
  $C$~\cite[Prop.~15.4]{MR1771927}. By standard ``way-out'' arguments (e.g.,
  \cite[Compl.~1.11.3.1]{MR2490557}) it now follows that $J_{f,M,N}$
  is a quasi-isomorphism for all $M\in \DQCOH^-(X)$, and the result
  follows.
\end{proof}
\begin{remark}
Note that if $A$ is an abelian variety and $\pi\colon BA\to \spec k$ is the
classifying stack, then $\RDERF\pi_*\colon \DQCOH(BA)\to \DCAT(\MOD(k))$ does
not admit a right adjoint. In fact, $BA$ is not concentrated (see
Section~\ref{S:coh-dim-BG}), so $\RDERF\pi_*$ does not preserve small
coproducts; thus, cannot be a left adjoint.
\end{remark}

\bibliography{references}
\bibliographystyle{bibstyle}
\end{document}